\documentclass[12pt]{amsart}

\usepackage{amssymb,amsmath}
\usepackage[dvips]{graphicx}

\newcounter{ENUM}
\newcommand{\be}{\begin{enumerate}}
\newcommand{\ee}{\end{enumerate}}
\newcommand{\beas}{\begin{eqnarray*}}
\newcommand{\eeas}{\end{eqnarray*}}
\newcommand{\bea}{\begin{eqnarray}}
\newcommand{\eea}{\end{eqnarray}}
\newcommand{\beq}{\begin{equation}}
\newcommand{\eeq}{\end{equation}}

\newcommand{\cals}{\mathcal{S}}
\newcommand{\sn}{\mathfrak{S}_n}
\newcommand{\pf}{\mathrm{PF}}
\newcommand{\pfn}{\mathrm{PF}_n}
\newcommand{\pfnx}{\mathrm{PF}_n(x)}

\newcommand{\fs}{\mathfrak{S}}

\newcommand{\sh}{\mathrm{SH}}
\newcommand{\shn}{\mathrm{SH}_n}
\newcommand{\shnx}{\mathrm{SH}_n(x)}
\newcommand{\ci}{{\langle -1\rangle}}

\newtheorem{theorem}{Theorem}[section]

\newtheorem{lemma}[theorem]{Lemma}
\newtheorem{corollary}[theorem]{Corollary}

\theoremstyle{definition}

\newtheorem{example}[theorem]{Example}

\theoremstyle{remark}

\numberwithin{equation}{section}

\def\cp{\mathcal{P}}

\def\nn{\mathbb{N}}

\newcommand{\qq}{\mathbb{Q}}

\begin{document}
\title{A Shifted Parking Function Symmetric Function} 

\date{\today}

\author{Richard P. Stanley}
\email{rstan@math.mit.edu}
\address{Department of Mathematics, University of Miami, Coral Gables,
FL 33124}

\maketitle

\begin{abstract}
We define a ``shifted analogue'' $\shn$ of the parking function symmetric
function $\pfn$. The expansion of $\shn$ in terms of three bases for
shifted symmetric functions is explicitly described. We don't know a
shifted analogue for parking functions themselves, but some desirable
properties of such an analogue are discussed.
\end{abstract}  

\section{Introduction}
We assume knowledge of symmetric functions such as may be found in
Macdonald \cite{macd} or Stanley \cite[Ch.~7]{ec2}. Given a symmetric
function $f(x)\in \Lambda_\qq$ (i.e., whose coefficients of monomials are
rational numbers) in variables $x=(x_1,x_2,\dots)$, define the
\emph{superfication} $f(x/y)=\omega_y f(x,y)$, where
$y=(y_1,y_2,\dots)$ is a new set of variables, and where $\omega_y$
denotes the involution $\omega$ acting on the $y$ variables only. We
then define the \emph{shiftification} of $f$ to be
$f(x/x)$. Equivalently, regarding $f$ as a polynomial in the power
sums $p_i$, we have
  \beq f(x/x) = f(p_{2i+1}\to 2p_{2i+1}, p_{2i}\to 0).
     \label{eq:fxx} \eeq
Thus $f(x/x)\in \Gamma:=\qq[p_1,p_3,p_5,\dots]$.

If $\lambda$ is a partition of $n$ into distinct parts, denoted
$\lambda\vDash n$, then let
$P_\lambda$ denote Schur's P-function indexed by $\lambda$, a shifted
analogue (but not the shiftification) of the ordinary Schur function 
$s_\lambda$ \cite[{\S}III.8]{macd}. In particular,
  $$ 2P_n(x) = h_n(x/x) = \sum_{k=0}^n e_k(x)h_{n-k}(x). $$
We also have the generating function
  \beq K(x,t):=1+2\sum_{n\geq 1}P_n(x)t^n = \prod_i
     \frac{1+x_it}{1-x_it}. \label{eq:pngf} \eeq
Moreover, the $P_{2n+1}$'s are algebraically independent and generate
$\Gamma$ as a $\qq$-algebra.
     
We now consider parking functions. A good survey of this topic was
given by C.\,H.\ Yan \cite{yan}. A \emph{parking function} of length
$n$ is a sequence $(a_1,\dots,a_n)$ of positive integers whose
increasing rearrangement $b_1\leq b_2\leq\cdots\leq b_n$ satisfies
$b_i\leq i$. Thus the symmetric group $\sn$ acts on the set $\cp_n$ of
all parking functions of length $n$ (the number of which is
$(n+1)^{n-1}$) by permuting coordinates. The Frobenius characteristic
symmetric function of this action is denoted $\pfn$, the \emph{parking
  function symmetric function}.

The symmetric function $\pfn$ has many remarkable properties. There
are simple product formulas for the coefficients when $\pfn$ is
expanded in terms any of the six bases $m_\lambda,
\mathrm{fo}_\lambda$, $h_\lambda, e_\lambda$, $p_\lambda, s_\lambda$,
where $\mathrm{fo}_\lambda$ denotes the forgotten symmetric function
$\omega(m_\lambda)$. Moreover, we have the generating function
  \beq \sum_{n\geq 0}\pfnx t^{n+1} = (tH(x,-t))^{\langle -1 \rangle},
     \label{eq:ci} \eeq
where
     \beq H(x,t)=\sum_{n\geq 0}h_n(x)t^n = \prod_i (1-x_it)^{-1},
       \label{eq:ht} \eeq
and $^{\langle -1\rangle}$ denotes compositional inverse with respect
to $t$. See for instance \cite{s-w}.

The main object of study in this paper is the shiftification of
$\pfn$, which we denote by $\shn$. Thus
  $$ \shnx = \pfn(x/x). $$
In Section~\ref{sec2} we give simple expansions of $\shn$ as linear
combinations of power 
sums $p_\lambda$ and the Schur P-functions $P_\lambda$. The proofs are
analogous to those for the expansion of $\pfn$ as a linear
combination of power sums and Schur functions, respectively. In
Section~\ref{sec:3} 
give the expansion of $\shn$ as a polynomial in $P_1,P_3,P_5,\dots$,
the analogue of expanding $\pfn$ as a polynomial in
$h_1,h_2,h_3,\dots$. The proof, however, is not analogous. It is
first necessary to express $P_{2n}$ in terms of the $P_{2i+1}$'s. The
most succinct way of expressing this relationship is the formula
  $$ \frac{-1+\sqrt{1+4A^2}}{2}=P_2t^2+P_4t^4+P_6t^6+\cdots, $$
where $A=P_1t+P_3t^3+P_5t^5+\cdots$.

What is currently missing from this development is a satisfactory
shifted analogue of parking functions themselves, not just their
corresponding symmetric functions. In the last section
(Section~\ref{sec:comb}) we discuss some desirable properties of such
an interpretation, as well as a ``naive shifted parking function''
which does not seem as interesting.

\section{The bases $p_\lambda$ and $P_\lambda$} \label{sec2}
In this section we discuss the expansion of $\shn$ in terms of the
power sums $p_\lambda$ (where $\lambda$ has odd parts) and the Schur
$P$-functions $P_\lambda$ (where $\lambda$ has distinct parts). As
mentioned in the previous section, the proofs are straightforward
variants of the corresponding results for $\pfn$.

\begin{theorem} \label{thm:powsum}
For $n\geq 1$ we have
  $$ \shn = \sum_{\substack{\lambda\vdash n\\ \lambda_i\,
    \mathrm{odd}}} z_\lambda^{-1}
    2^{\ell(\lambda)}(n+1)^{\ell(\lambda)-1} p_\lambda, $$
where $\ell(\lambda)$ denotes the length (number of $\lambda_i>0$) of
$\lambda$. 
\end{theorem}  

\begin{proof}
We have \cite[Prop.~1.1]{s-w}\cite[Exer.~7.48(f)]{ec2} (note that in
this latter exercise, we have $\pfn=\omega F_{\mathrm{NC}_{n+1}}$)
  $$ \pfn = \sum_{\lambda\vdash n} z_\lambda^{-1}
    (n+1)^{\ell(\lambda)-1} p_\lambda. $$
Now use equation~\eqref{eq:fxx} and the definition $\shn(x) =
\pf_n(x/x)$.
\end{proof}

If $f$ is a homogeneous symmetric function of degree $n$, we define
its \emph{dimension} by $\dim f=\langle f,p_1^n\rangle$. If $f$ is the
Frobenius characteristic $\mathrm{ch}(\chi)$ of a character $\chi$ of
$\sn$, then $\dim f = \dim \chi = \chi(\mathrm{id})$. 

\begin{corollary} \label{cor:total}
We have $\dim \shn = 2^n (n+1)^{n-1}$. 
\end{corollary}  

\begin{proof}
Follows from Theorem~\ref{thm:powsum} and the orthogonality relation
$\langle p_\mu,p_\lambda\rangle = z_\lambda\delta_{\mu,\lambda}$
\cite[Prop.~7.9.3]{ec2}. 
\end{proof}  

For the $P_\lambda$ expansion we need the following fundamental lemma,
the shifted analogue of equation~\eqref{eq:ci}.

\begin{lemma} \label{lemma:fund}
We have
  \beq \sum_{n\geq 0}\shn(x) t^{n+1} =\left( tK(x,-t)\right)^{\langle
    -1\rangle}, \label{eq:sshn} \eeq 
where $^{\langle -1\rangle}$ denote compostional inverse with respect
to $t$, and where $K(x,t)$ is defined in equation~\eqref{eq:pngf}. 
\end{lemma}

\begin{proof}
Shiftification is an algebra homomorphism, from
which it follows that it commutes with compositional inverse with
respect to $t$. That is, if $F(x,t)=\sum_{n\geq 0}f_n(x)t^{n+1}$, where
$f_n(x)\in\Lambda_\qq$ and $f_0(x)=1$, and if
$G(x,t)=F(x,t)^{\langle -1\rangle}$ (taken with respect to $t$), then
$F(x/x,t)^{\langle -1\rangle} =G(x/x,t)$. Now with $H(x,t)$ as in
equation~\eqref{eq:ht} we have
    \beas H(x/x,t) & = &  \sum_{n\geq 0}
      \left( \sum_{k=0}^n e_k h_{n-k} \right) t^n\\ & = &
    \left( \sum_{n\geq 0}e_nt^n\right)\left( \sum_{n\geq
      0}h_nt^n\right)\\ & = &
    \prod_i \frac{1-x_it}{1+x_it}, \eeas
Now replace $t$ by $-t$ and shiftify
equation~\eqref{eq:ci}. 
\end{proof}

The appearance of a compositional inverse in Lemma~\ref{lemma:fund}
suggests the Lagrange inversion formula could be useful. We use it in
the following form \cite[Thm.~5.4.2]{ec2}. Let $F(t)=t+\sum_{n\geq 2}
a_nt^n$ be a power series over a field of characteristic 0.\ Write
$[t^m]G(t)$ for the coefficient of $t^m$ in the power series
$G(t)$. Then for any $n\in\nn:=\{0,1,2,\dots\}$,
    \beq [t^{n+1}]F^{\langle -1\rangle}(t) = \frac{1}{n+1} [t^n] \left(
    \frac{t}{F(t)}\right)^{n+1}. \label{eq:lg2} \eeq

The next result gives the $P_\lambda$ expansion of $\shn$. We use the
notation $f(1^{n+1})$ to indicate that in $f(x)$ we set
$x_1=\cdots=x_{n+1}=1$ and $x_i=0$ for $i>n+1$. 

\begin{theorem} \label{thm:pl}
We have
  $$ \shnx = \frac{1}{n+1}\sum_{\lambda\vDash n} 2^{\ell(\lambda)}
    P_\lambda(1^{n+1})P_\lambda(x). $$
\end{theorem}

\begin{proof}
By Lemma~\ref{lemma:fund} we have
  $$ \shnx = [t^{n+1}]\left( t\prod_i\frac{1-x_it}{1+x_it}
     \right)^\ci. $$
By equation~\eqref{eq:lg2} it follows that
  \beq \shnx = \frac{1}{n+1}[t^n]\left(
    \prod_i\frac{1+x_it}{1-x_it}\right)^{n+1}. \label{eq:shnlg} \eeq
The Cauchy identity for $P_\lambda$ \cite[(8.13)]{macd} states that
 \beq \sum_\lambda 2^{\ell(\lambda)}P_\lambda(x)P_\lambda(y) =
   \prod_i \frac{1+x_iy_j}{1-x_iy_j}, \label{eq:cauchy} \eeq
where $\lambda$ ranges over all partitions with distinct
parts. Set $y_1=\cdots=y_{n+1}=1$ and $y_i=0$ for $i>n+1$ and compare
with equation~\eqref{eq:shnlg} to complete the proof.
\end{proof}

\section{The basis $P_{\lambda_1}P_{\lambda_2}\cdots$} \label{sec:3}
In this section we write $\shn$ as a polynomial in the Schur
$P$-functions $P_1,P_3,P_5,\dots$. Before doing so, let us note that
it is very easy to write $\shn$ as a polynomial in $P_1, P_2, P_3,
\dots$ (not unique since the $P_i$'s are not algebraically
independent). For any $\lambda\vdash n$ define
  $$ V_\lambda = P_{\lambda_1}P_{\lambda_2}\cdots. $$ 
Then we can shiftify the formula
\cite[Exer.~7.48(f)]{ec2}\cite[(1.1)]{s-w} 
 $$ \pfn = \sum_{\lambda\vdash n} \frac{n(n-1)\cdots
   (n-\ell(\lambda)+2)}{m_1!\cdots m_n!}h_\lambda, $$
where $\lambda$ has $m_i$ parts equal to $i$. Since $h_i(x/x) =
2P_i(x)$, we obtain
  \beq \shn =\sum_{\lambda\vdash n}
   \frac{2^{\ell(\lambda)}n(n-1)\cdots
   (n-\ell(\lambda)+2)}{m_1!\cdots m_n!}V_\lambda. \label{eq:easy}
   \eeq 
   
We now turn to the more difficult problem of writing $\shn$ as a
polynomial in $P_1,P_3,P_5,\dots$. This representation is unique since
the $P_{2i+1}$'s are algebraically independent. The first step is to
write $P_{2n}$ as a polynomial in $P_1,P_3,P_5,\dots$.

\begin{lemma} \label{lemma:a}
Let $A=P_1t+P_3t^3+P_5t^5+\cdots$. Then
     $$ P_2t^2+P_4t^4+P_6t^6+\cdots=\frac{-1+\sqrt{1+4A^2}}{2}. $$     
\end{lemma}

\begin{proof}
  We have $K(x,t)=1+2A+2B$ and $K(x,-t)=1-2A+2B$. Note also that
  $$ K(x,-t)=\frac{1}{K(x,t)} =\frac{1}{1+2A+2B}. $$
Hence 
  $$ 1-2A+2B = \frac{1}{1+2A+2B}. $$
Solving for $B$ gives
  $$ B=\frac{-1\pm \sqrt{1+4A^2}}{2}. $$
The correct sign is plus.
\end{proof}

It is now easy to give an explicit formula for $P_{2n}$ as a
polynomial in $P_1,P_3,P_5,\dots$.

\begin{corollary}
We have
  $$ P_{2n}=\sum_{\substack{\lambda\vdash 2n\\ \lambda_i\,
    \mathrm{odd}}} (-1)^{\frac 12(\ell(\lambda)-2)}
   C_{\frac 12(\ell(\lambda-2))}\binom{\ell(\lambda)}{m_1,m_3,\dots}
   V_\lambda, $$
where $C_k$ denotes a Catalan number, and where $\lambda$ has $m_i$
parts equal to $i$.   
 \end{corollary}

\begin{proof}
The well-known generating function for Catalan numbers (e.g.,
\cite[\S6.2]{ec2}) implies
   $$ \frac{-1+\sqrt{1+4A^2}}{2} = \sum_{k\geq 0}(-1)^{k-1}C_k
     A^{2k+2}. $$
Now by the multinomial theorem,
   \beas A^{2k+2} & = & (P_1t+P_3t^3+P_5t^5+\cdots)^{2k+2}\\ & = &
   \sum_{n\geq 1}\sum_{\substack{\lambda\vdash n\\
     \ell(\lambda)=2k+2\\ \lambda_i\,\mathrm{odd}}}
   \binom{2k+2}{m_1,m_2,\dots}V_\lambda t^n,
   \eeas
and the proof follows.       
\end{proof}

\begin{lemma} \label{lemma:taylor}
We have the Taylor series expansion 
  $$ (2z+\sqrt{1+4z^2})^{n+1} = 
   1+\sum_{k\geq 1} (n+1)g_k(n) \frac{z^k}{k!}, $$
where
  \beq g_k(n)=2^k (n+k-1)(n+k-3)(n+k-5)\cdots(n-k+3).
     \label{eq:gkn} \eeq
  \end{lemma}

\begin{proof}
This result is surely well-known or equivalent to a well-known
result, so we just give the idea of a proof. It is very similar to
\cite[Problem~A32(a)]{cat}. The function $f(z)=z(2z+\sqrt{1+4z^2})$
has compositional inverse
  $$ f^{\langle -1\rangle}(z) = \frac{z}{\sqrt{1+4z}}. $$
The proof follows easily from equation~\eqref{eq:lg2}.
\end{proof}

We can now prove the main result of this section.

\begin{theorem} \label{thm:main}
We have
  $$ \hspace{-10em} \shn = $$
  $$\sum_{\substack{\lambda\vdash n\\ \lambda_i\,
    \mathrm{odd}}} \frac{2^{\ell}}{\ell!}
    \binom{\ell}{m_1,m_3,\dots}(n+\ell-1)
    (n+\ell-3)\cdots (n-\ell+3)V_\lambda,
    $$
where $\ell=\ell(\lambda)$ and $\lambda$ has $m_i$ parts equal to $i$.
\end{theorem}

\begin{proof}
By Lemma~\ref{lemma:a} we have
   $$ K(x,t)= 1+2A+2\left( \frac{-1+\sqrt{1+4A^2}}{2}\right) =
    2A+\sqrt{1+4A^2}. $$
Hence by equation~\eqref{eq:shnlg} and Lemma~\ref{lemma:taylor},
  \beas \shnx & = & \frac{1}{n+1}[t^n](2A+\sqrt{1+4A^2})^{n+1}\\
  & = & [t^n] \sum_{k\geq 1}g_k(n) \frac{A^k}{k!}, \eeas
where $g_k(n)$ is given by equation~\eqref{eq:gkn}.   
Now expand each $A^k$ by the multinomial theorem and extract the
coefficient of $t^n$,
\end{proof}

Let $r_n$ denote a large Schr\"oder number (see
e.g.\ \cite[A006318]{oeis}\cite[Exer.~6.39]{ec2}). They play a role
for $\shn$ similar to Catalan numbers for
$\pfn$. Corollary~\ref{cor:schr} below gives some occurrences of
$r_n$, but first we need a lemma.

\begin{lemma} \label{lemma:sumcoef}
Let $f\in \Gamma$, and assume that $f$ is homogeneous of degree
$n$. Consider the expansions
  $$ f = \sum_{\substack{\lambda\vdash n\\
      \lambda_i\,\mathrm{odd}}} a_\lambda V_\lambda =
    \sum_{\lambda\vDash n} b_\lambda P_\lambda. $$
Then $\sum_\lambda a_\lambda = b_n$, where $b_n$ is short for
$b_{(n)}$. 
\end{lemma}  

\begin{proof}
Let $f(1)= f(x_1=1,
x_2=x_3=\cdots=0)$. Thus the map $f(x)\mapsto f(1)$ is an algebra
homomorphism $\Gamma\to\qq$.  Putting $x=(1,0,0,\dots)$ in
equation~\eqref{eq:pngf} shows that $P_n(1)=1$ for all $n\geq
1$. Hence $V_\lambda(1)=1$ for all partitions $\lambda$. It follows
that $f(1)$ is the sum of the coefficients when $f(x)$ is written as a
polynomial in $P_1,P_3,P_5,\dots$.

Let $\lambda$ be a partition with distinct parts. We claim that
   $$ P_\lambda(1)=\left\{ \begin{array}{rl} 1, & \lambda=(n)\\
       0, & \mathrm{otherwise}. \end{array} \right\} $$
One way to see this is to put $x=(1,0,0,\dots)$ in the
shifted Cauchy identity (equation~\eqref{eq:cauchy}) and compare with
equation~\eqref{eq:pngf}. Another proof follows from the combinatorial
interpretion of $P_\lambda$ in terms of shifted Young tableaux (e.g.,
\cite[8.16]{macd}). It follows that $f(1)$ is the coefficient of $P_n$
when $f(x)$ is written as a linear combination of $P_\lambda$'s, where
$\lambda\vDash n$, and the proof follows.
\end{proof}
  
\begin{corollary} \label{cor:schr}
The following numbers are equal to $r_n$.
 \begin{itemize}
   \item $\frac{2}{n+1}P_n(1^{n+1})$
   \item the coefficient of $P_n$ when $\shn$ is written as a linear
     combination of the $P_\lambda$'s, where $\lambda\vDash n$
   \item the sum of the coefficients of $\shn$ when written as a
     polynomial in the $P_k$'s ($k$ odd) 
 \end{itemize}
\end{corollary}

\begin{proof}
From equation~\eqref{eq:pngf} we have
  \beq 1+2\sum_{n\geq 1}P_n(1^{n+1})t^n =
    \left(\frac{1+t}{1-t}\right)^{n+1}. \label{eq:lg1} \eeq
In equation~\eqref{eq:lg2}
set $F(t) = t(1-t)/(1+t)$. By equation~\eqref{eq:lg1} the right-hand
side of equation~\eqref{eq:lg2} becomes $2P_n(1^{n+1})/(n+1)$. Now
   \beas \left( \frac{t(1-t)}{1+t}\right)^{\langle -1\rangle} & = &
     \frac{1-t-\sqrt{1-6t+t^2}}{2}\\ & = & \sum_{n\geq 0} r_n t^{n+1},
       \eeas
and the proof of the first item follows.

The second item follows immediately from the first item and
Theorem~\ref{thm:pl}.

Lemma~\ref{lemma:sumcoef} shows that items two and three are equal, so
the proof follows.
\end{proof}


The third item in Corollary~\ref{cor:schr} suggests the following
question. By Theorem~\ref{thm:main} the coefficients of $\shn$, when
written as a polynomial in the $P_k$'s ($k$ odd), are nonnegative. Do
these coefficients have a combinatorial interpretation refining some
combinatorial interpretation of $r_n$? Here is a table of some of
these polynomials:
   \beas \sh_1 & = & 2P_1\\ \sh_2 & = & 6P_1^2\\
    \sh_3 & = & 20P_1^3 + 2P_3\\ 
    \sh_4 & = & 70P_1^4 + 20P_3P_1\\ 
    \sh_5 & = & 252P_1^5+140P_3P_1^2+2P_5\\ 
    \sh_6 & = & 924P_1^6 + 840P_3P_1^3 +28P_5P_1 + 14P_3^2. \eeas
Note that the coefficient of $P_1^n$ in $\shn$ is $\binom{2n}{n}$, a
special case of Theorem~\ref{thm:main}.
 
\section{A combinatorial interpretation?} \label{sec:comb}
As mentioned in the Introduction, the one feature missing from this
development is the concept of a shifted parking function. We can give
a ``naive'' definition based on equation~\eqref{eq:easy}, but, as we
discuss below, what is really wanted is a definition based on
Theorem~\ref{thm:main}.

The naive definition is the following. A \emph{naive shifted parking
  function} (NSPF) $\pi$ of length $n$ is an ordinary parking function of
length $n$ with each term colored red or blue. Since there are
$(n+1)^{n-1}$ ordinary parking functions of length $n$, there are
$2^n(n+1)^{n-1}$ NSPF's of length $n$, which is what we want by
Corollary~\ref{cor:total}.

By equation~\eqref{eq:easy} we would like to partition the set
$\mathcal{N}_n$ of NSPF's of length $n$ into $r_n$ blocks such that
(a) each block $B$ is associated with a partition
$\lambda=\lambda(B)\vdash n$, (b) the size of $B$ is given by
  $$ \#B = \dim V_\lambda = 2^{n-\ell(\lambda)}
    \binom{n}{\lambda_1,\lambda_2,\dots}, $$
and (c) the number of blocks $B$ which correspond to $\lambda$ is
equal to the coefficient of $V_\lambda$ in the right-hand side of
equation~\eqref{eq:easy}. We can do this as follows: a block $B$
consists of all NSPF's $\sigma$ with specified part multiplicities
(i.e., we specify for each $i$ the number $a_i$ of $i$'s in $\sigma$)
and specified colors for the first (leftmost) occurrence of each part
equal to $i$ ($1\leq i\leq n$). The partition $\lambda$ consists of the
$a_i$'s arranged in weakly decreasing order.

\begin{example}
Let $n=2$. There are 12 NSPF's of length two. The blocks are listed
below. We use an overline for the color red and no overline for blue.
  $$ \{ 11, 1\bar{1}\},\ \{\bar{1}1,\bar{1}\bar{1}\},\
    \{12,21\},\ \{\bar{1}2,2\bar{1}\},\ \{1\bar{2},\bar{2}1\},\
    \{\bar{1}\bar{2},\bar{2}\bar{1}\} $$
All the blocks have size two since $\dim P_1^2=\dim P_2=2$.    
\end{example}
  
%

We now consider ``serious'' shifted parking functions.
Let
$\cals_n$ be the (putative) set of all shifted parking functions of
``size'' $n$. It should have the following property: there is a
partition $\psi$ of $\cals_n$ such that each block $B$ corresponds
in a natural way to a partition $\lambda$ of $n$ into odd parts. The
number of elements of $B$ is the coefficient of
$P_{\lambda_1}P_{\lambda_2}\cdots$ when $\shn$ is written as a
polynomial in the $P_i$'s for $i$ odd (given by
Theorem~\ref{thm:main}). The size $\#B$ of $B$ is equal to
  $$ \dim V_\lambda =
2^{n-\ell(\lambda)}\binom{n}{\lambda_1,\lambda_2,\dots}. $$ This
implies that the total number of blocks is $r_n$ and that
$\#\cals_n=2^n(n+1)^{n-1}$. The fundamental difference with NSPF's
$\sigma$ is that the partition $\lambda$ associated
with a shifted parking function has only odd parts.

Also desirable would be a naive shifted analogue or shifted analogue
of the action of $\sn$ on parking functions. Ideally we would have a
suitable action of the double cover (for $n\geq 4$)
$\widetilde{\fs}_n$ of $\sn$ on the complex vector space with basis
$\mathcal{N}_n$ or $\cals_n$. Probably this is too much to hope for.

\end{document}